\documentclass{siamart190516}
\usepackage[utf8]{inputenc}
\usepackage{graphicx}
\usepackage{epsfig}
\usepackage{amssymb}
\usepackage{fullpage}
\usepackage{amsmath}
\usepackage{color}
\usepackage{subfigure}
\usepackage[boxed,noline,algo2e]{algorithm2e}
\usepackage{multirow}

\usepackage{lipsum}
\usepackage{amsfonts}
\usepackage{graphicx}
\usepackage{epstopdf}
\usepackage{algorithmic}
\usepackage{gensymb}

\usepackage{cleveref}

\graphicspath{{images/}}
\DeclareGraphicsExtensions{.png}

\begin{document}

\title{On the recursive structure of multigrid cycles
\thanks{Research supported by the Israel Science Foundation, Grant No. 1639/19.}}


\author{Or Avnat\thanks{ Department of Computer Science,
Technion---Israel Institute of Technology, Haifa, Israel.
(\email{oravnat@gmail.com}, \email{irad@cs.technion.ac.il}).}
\and Irad Yavneh\footnotemark[2]}

\ifpdf
\hypersetup{
  pdftitle={On the recursive structure of multigrid cycles},
  pdfauthor={O. Avnat and I. Yavneh}
}
\fi

\date{\today}
\maketitle

\begin{abstract}
A new fixed (non-adaptive) recursive scheme for multigrid algorithms is introduced. Governed by a positive parameter $\kappa$ called the cycle counter, this scheme generates a family of multigrid cycles dubbed $\kappa$-cycles. The well-known $V$-cycle, $F$-cycle, and $W$-cycle are shown to be particular members of this rich $\kappa$-cycle family, which satisfies the property that the total number of recursive calls in a single cycle is a polynomial of degree $\kappa$ in the number of levels of the cycle. This broadening of the scope of fixed multigrid cycles is shown to be potentially significant for the solution of some large problems on platforms, such as graphics processing units, where the overhead induced by numerous sequential calls to the coarser levels may be relatively significant. In cases of problems for which the convergence of standard $V$-cycles or $F$-cycles (corresponding to $\kappa=1$ and $\kappa=2$, respectively) is particularly slow, and yet the cost of $W$-cycles is very high due to the large number of coarse-level calls (which is exponential in the number of levels), intermediate values of $\kappa$ may prove to yield significantly faster run-times. This is demonstrated in examples where $\kappa$-cycles are used for the solution of rotated anisotropic diffusion problems, both as a stand-alone solver and as a preconditioner. Moreover, a simple model is presented for predicting the approximate run-time of the $\kappa$-cycle, which is useful in pre-selecting an appropriate cycle counter for a given problem on a given platform. Implementing the $\kappa$-cycle requires making just a small change in the classical multigrid cycle.
\end{abstract}



%

\section{Introduction} \label{sec:Introduction}

Multigrid methods are especially well-known for their efficiency in solving linear systems arising from the discretization of partial differential equations (PDEs) \cite{Bra77,Bra82,BHM00,RDF06,Hac85,TOS01,Yav06}. Over many years, multigrid methods have been developed and their scope greatly expanded, and numerous multigrid algorithms now exist. One aspect that seems to have received relatively limited attention in these developments is the recursive structure of the multigrid cycle. Although adaptive multigrid cycles have certainly been studied (e.g., in \cite{Bra73,Bra82,Rue93}), in the vast majority of applications a fixed recursive strategy is employed: most commonly the so-called $V$-cycle, and often $W$-cycles or $F$-cycles. Other cycling strategies have been employed over the years for specific purposes, but they are far less common. Indeed, multigrid textbooks typically introduce the standard cycle, with the {\em cycle index}, typically denoted by $\gamma$, as a parameter that determines the number of recursive calls to the cycle routine at each level of the multigrid hierarchy, and from then on they refer only to $\gamma = 1$ or $\gamma = 2$ (corresponding to the $V$ and $W$ cycles, respectively). The $F$-cycle, which in a sense mixes $\gamma = 1$ and $\gamma = 2$, is also mentioned because it is useful in many cases. For example, the classical introductory textbook \cite{BHM00} presents the multigrid cycle with the cycle index, denoted there by $\mu$, and then mentions that, in practice, only $\mu = 1$ and $\mu = 2$ are used, corresponding to the $V$ and $W$ cycles, respectively. As in other textbooks, alternative fixed cycling strategies are not considered.

To formally define the classical multigrid cycles, we consider the solution of a linear system of the form
$$
A^1u^1 = f^1 \, ,
$$

\noindent
where the $1$ superscript indicates that this problem corresponds to the finest level of the multigrid hierarchy. The standard family of multigrid cycles with cycle index $\gamma$ and $n$ levels is defined recursively in Algorithm \ref{alg:gamma_cycle} (see, e.g., \cite{BHM00}).

\SetNlSkip{-0.5em}

\begin{algorithm2e}[h]
\DontPrintSemicolon
\label{alg:gamma_cycle}
\caption{The classical $\gamma$-cycle}
{$v^\ell \leftarrow \gamma\mbox{-}cycle(v^\ell,f^\ell,A^\ell,\ell,n,\gamma)$}\;
\Indp
\nl
If $\ell == n$ (coarsest level), solve $A^\ell v^\ell = f^\ell$ and return.\;
\nl
{\em Relax} $\nu_1$ times on $A^\ell u^\ell = f^\ell$ with initial guess $v^\ell$.\;
\nl
$f^{\ell+1} \leftarrow Restrict(f^\ell - A^\ell v^\ell)$.\;
\nl
$v^{\ell+1} \leftarrow 0$. \;
\nl
Repeat $\gamma$ times:\;
$v^{\ell+1} \leftarrow \gamma\mbox{-}cycle(v^{\ell+1},f^{\ell+1},A^{\ell+1},\ell+1,n,\gamma)$.\;
\nl
$v^\ell \leftarrow v^\ell + Prolong(v^{\ell+1})$.\;
\nl
{\em Relax} $\nu_2$ times on $A^\ell u^\ell = f^\ell$ with initial guess $v^\ell$.\;
\end{algorithm2e}

As noted above, with $\gamma = 1$ (i.e., a single recursive call per cycle) we get the classical $V$-cycle, while $\gamma = 2$ (two recursive calls) yields the $W$-cycle. The third commonly used cycle, the $F$-cycle, cannot be described in this form. Rather awkwardly, once we have defined the $\gamma$-cycle, the $F$-cycle 
is defined in Algorithm \ref{alg:F_cycle}. Here, we have a single recursive call to the $F$-cycle, followed by a call to the $V$-cycle (i.e., the $\gamma$-cycle with $\gamma = 1$).

We remark that, when $\ell = n-1$, the first recursive call in line 5 of Algorithm \ref{alg:gamma_cycle} returns the exact solution of the coarsest-level problem, so the execution of subsequent recursive calls when $\gamma > 1$ has no effect. Similarly, line 6 of Algorithm \ref{alg:F_cycle} has no effect when $\ell = n-1$. Nevertheless, we adopt this standard form for all levels. This choice typically has a negligible effect on the overall cost of the cycle, and furthermore, the additional call to the coarsest level does have an effect when the problem on line 1 is only solved approximately, which is not uncommon in real applications.

\begin{algorithm2e}[h]
\DontPrintSemicolon
\label{alg:F_cycle}
\caption{The F-cycle}
{$v^\ell \leftarrow \mbox{$F$-}cycle(v^\ell,f^\ell,A^\ell,\ell,n)$}\;
\Indp
\nl
If $\ell == n$ (coarsest level), solve $A^\ell v^\ell = f^\ell$ and return.\;
\nl
{\em Relax} $\nu_1$ times on $A^\ell u^\ell = f^\ell$ with initial guess $v^\ell$.\;
\nl
$f^{\ell+1} \leftarrow Restrict(f^\ell - A^\ell v^\ell)$.\;
\nl
$v^{\ell+1} \leftarrow 0$. \;
\nl
$v^{\ell+1} \leftarrow \mbox{$F$-}cycle(v^{\ell+1},f^{\ell+1},A^{\ell+1},\ell+1,n)$.\;
\nl
$v^{\ell+1} \leftarrow \gamma\mbox{-}cycle(v^{\ell+1},f^{\ell+1},A^{\ell+1},\ell+1,n,1)$.\;
\nl
$v^\ell \leftarrow v^\ell + Prolong(v^{\ell+1})$.\;
\nl
{\em Relax} $\nu_2$ times on $A^\ell u^\ell = f^\ell$ with initial guess $v^\ell$.\;
\end{algorithm2e}

The reason for the casual treatment of the recursive structure of multigrid algorithms is probably that there is often no practical need for alternative strategies. If the simplest and cheapest variant---the $V$-cycle---does the job, as it often does, then there seems to be no reason to look any further. If a stronger cycle is required, then we can try $\gamma = 2$, i.e., the $W$-cycle, or compromise with the $F$-cycle if that is efficient. Nevertheless, in this paper we argue that, although these three cycle types
may have been sufficient in the days of more limited machines and moderate-sized problems, it is worthwhile to reconsider this in modern-day and future computing platforms. Such platforms may have a very high throughput for big problems, but are much less efficient for a large number of sequential small problems. This observation is true for the GPU platforms we consider in this paper. More generally, efficient parallel implementation of algebraic multigrid (AMG) for unstructured problems on modern high-performance computers is known to be challenging \cite{baker2012preparing}, requiring specialized handling. Examples of this include explicit sparsification of coarse-grid operators \cite{BFG16}, gathering and redistributing coarse-grid data to reduce communication costs \cite{GGJ13}, and non-Galerkin coarsening \cite{FS14,TY15}. Even structured multigrid solvers require special redistribution techniques on coarse levels to obtain good scaling properties for very large problems \cite{ROM18}.     

Suppose that we apply the standard multigrid algorithm with $n$ levels, numbered $1$ to $n$ from finest to coarsest. In a single $V$-cycle, implemented recursively, the cycle routine is called only once per level---a total of $n$ calls. However, in the $W$-cycle it is called $2^{\ell-1}$ times on level $\ell$, $\ell = 1, \ldots \, , n$, totaling $2^n-1$ calls to the cycle routine in one complete $W$-cycle. For example, in a 2D problem with standard coarsening, the cycle routine is called once per iteration at the finest level, where there are $N$ variables. At the next level the $W$-cycle is called twice and the number of variables is N/4. Each such $W$-cycle makes two calls at the next-coarser level, with the number of variables divided again by 4. Thus, the number of activations of the cycle is doubled with each coarsening, even as the size of the problem is divided by 4. The upshot is that, if the number of operations per level depends linearly on the number of variables, then the total number of operations remains linear in $N$. However, the calls to the cycle routine are performed sequentially, and their number is evidently exponential in the number of levels. For instance, for such a problem of size $2^{10} \times 2^{10}$, about one million fine-level variables, a $W$-cycle using ten levels performs over 1000 calls to the cycle routine, compared to just ten calls for a $V$-cycle, whereas the number of operations performed by the $W$-cycle is only $50\%$ more compared to the $V$-cycle.   

This ``exponential gap'' between $V$ and $W$-cycles in the number of routine calls may be significant, especially when $(a)$ big problems are attacked; $(b)$ the total cost of coarse-grid work is relatively high because of complexity growth on coarser grids, as may happen in AMG methods which provide limited explicit control over the complexity of the operators on the coarse grids and over the rate of coarsening; $(c)$ parallel processing efficiency is significantly reduced in the $W$-cycle, because the coarse-grid problems are small and the $2^{n}-1$ cycle routine calls are performed sequentially, and not in parallel. 

Finally, on a more conceptual level, we ask: why define a general parameter---the cycle index $\gamma$---if only the values 1 or 2 are used in practice? Moreover, why is the popular $F$-cycle not describable in the standard cycling-scheme framework? 
These aspects of the standard multigrid cycle seem to be somewhat at odds with the spirit of Occam's Razor, which gives preference to simplicity of representation. 
In the next section we replace the {\em cycle index} $\gamma$ by another positive integer, $\kappa$, which we dub the {\em cycle counter} in order to distinguish it from the standard cycle index. With this we define a family of multigrid cycles whose recursive structure is determined by $\kappa$. For certain choices of $\kappa$, we obtain the three common cycles, $V$, $W$ and $F$, but for other choices we get other cycles, that are stronger than the $V$ and $F$ cycles, and yet retain the property that the total number of cycle routine calls over all levels\footnote{Although the algorithms are presented in recursive form, our arguments hold also for non-recursive implementations, which are readily available. The higher computational cost resulting from increasing $\kappa$ is due to the additional sequential traversing across levels, rather than to the recursion per se.} is polynomial in $n$, rather than exponential as in the $W$-cycle. 

The remainder of this paper is organized as follows. In \cref{sec:TheKappaCycle} we introduce the $\kappa$-cycle and present theoretical and practical complexity properties. In  \cref{sec:timePrediction} we introduce and verify in practice a simple model for predicting the run-time of $\kappa$-cycles. In  \cref{sec:NumericalTests} we test the $\kappa$-cycle performance on GPU processors, and we summarize and conclude in \cref{sec:Conclusions}.

\section{The \texorpdfstring{{$\kappa$}}{k}-cycle} \label{sec:TheKappaCycle}

\subsection{\texorpdfstring{{$\kappa$}}{k}-cycle definition} \label{subsec:RecursiveStructure}
Given a positive integer $\kappa$, the $\kappa$-cycle, defined in Algorithm \ref{alg:kappa_cycle}, performs one recursive call inheriting the same counter $\kappa$, followed by a second call with the counter reduced by one. The latter call is performed only if the counter for this call remains positive.
\begin{algorithm2e}[h]
\DontPrintSemicolon
\label{alg:kappa_cycle}
\caption{The $\kappa$-cycle}
{$v^\ell \leftarrow \kappa\mbox{-}cycle(v^\ell,f^\ell,A^\ell,\ell,n,\kappa)$}\;
\Indp
\nl
If $\ell == n$ (coarsest level), solve $A^\ell v^\ell = f^\ell$ and return.\;
\nl
{\em Relax} $\nu_1$ times on $A^\ell u^\ell = f^\ell$ with initial guess $v^\ell$.\;
\nl
$f^{\ell+1} \leftarrow Restrict(f^\ell - A^\ell v^\ell)$.\;
\nl
$v^{\ell+1} \leftarrow 0$. \;
\nl
$v^{\ell+1} \leftarrow \kappa\mbox{-}cycle(v^{\ell+1},f^{\ell+1},A^{\ell+1},\ell+1,n,\kappa)$.\;
\nl
If $\kappa > 1$\;
$v^{\ell+1} \leftarrow \kappa\mbox{-}cycle(v^{\ell+1},f^{\ell+1},A^{\ell+1},\ell+1,n,\kappa - 1)$.\;
\nl
$v^\ell \leftarrow v^\ell + Prolong(v^{\ell+1})$.\;
\nl
{\em Relax} $\nu_2$ times on $A^\ell u^\ell = f^\ell$ with initial guess $v^\ell$.\;
\end{algorithm2e}

\noindent \cref{fig:cycle_illustrations} provides an illustration of $\kappa$-cycles with 5 levels. Circles represent relaxation, corresponding to lines 2 and 8 in Algorithm \ref{alg:kappa_cycle}, while each red X represents a recursive call.

The following proposition relates the $\kappa$-cycle to the three classical multigrid cycles.

\begin{proposition} \label{prop:VFW}
$ $
\begin{itemize}
\item
For $\kappa = 1$, the $\kappa$-cycle is identical to the standard $V$-cycle.
\item
For $\kappa = 2$, the $\kappa$-cycle is identical to the standard $F$-cycle.
\item
For $\kappa \geq n$, the $\kappa$-cycle is identical to the standard $W$-cycle.
\end{itemize}
\end{proposition}

\begin{proof}
For $\kappa = 1$, Algorithm \ref{alg:kappa_cycle} evidently performs just a single recursive call on each level, so it is indeed identical to the $V$-cycle. Given this, observe that for $\kappa = 2$, Algorithm \ref{alg:kappa_cycle} performs a recursive call with $\kappa = 2$, followed by a $V$-cycle, so it is indeed identical to the $F$-cycle of Algorithm \ref{alg:F_cycle}. For the third claim, observe that the smallest cycle counter that appears in a cycle routine call at level $\ell>1$ is smaller by one than the smallest cycle counter that appears in a cycle routine call at the next-finer level $\ell-1$. If the finest-level $\kappa$ is at least $n$, this implies that $\kappa$ is positive in the entire cycle, implying two recursive calls on each level but the coarsest. Therefore, for $\kappa \geq n$, the $\kappa$-cycle is equivalent to the $\gamma$-cycle of Algorithm \ref{alg:gamma_cycle} with $\gamma = 2$, i.e., the $W$-cycle.
\end{proof}

\begin{figure}
\centering
\subfigure[1-cycle.] {\includegraphics[scale=0.4]{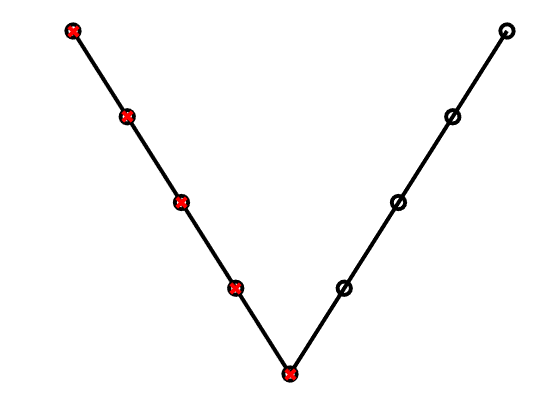}}
\subfigure[2-cycle.] {\includegraphics[scale=0.4]{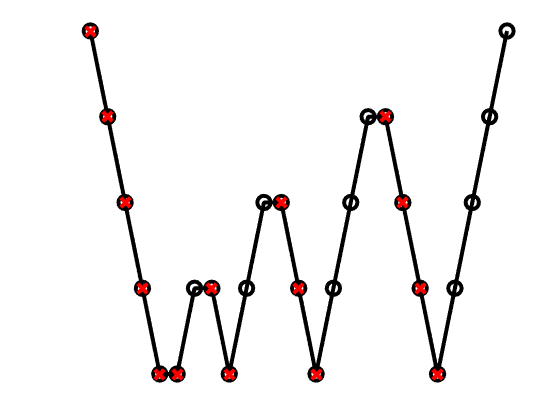}}
\subfigure[3-cycle.] {\includegraphics[scale=0.4]{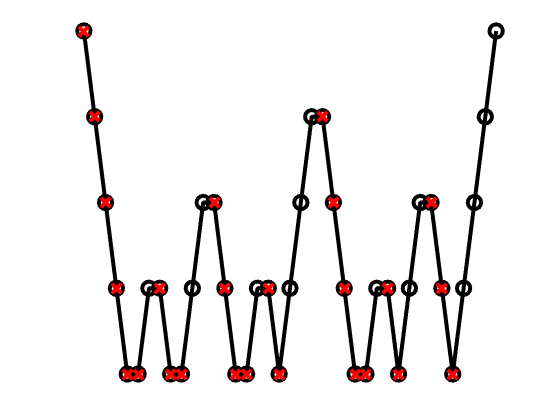}}
\subfigure[5-cycle.] {\includegraphics[scale=0.4]{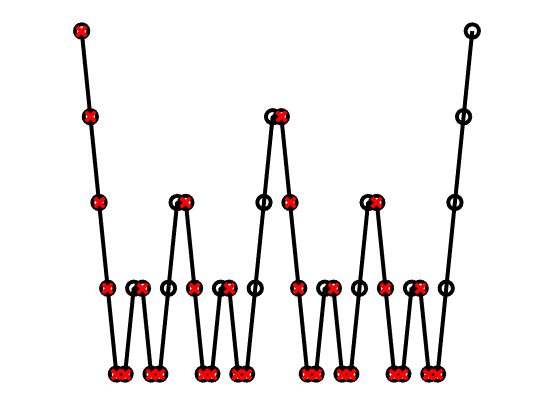}}
\caption{The $\kappa$-cycle is illustrated with 5 levels and $\kappa=1,2,3,5$. Circles denote relaxations, corresponding to lines 2 and 8 of Algorithm \ref{alg:kappa_cycle}, while red X's denote recursive calls.}
\label{fig:cycle_illustrations}
\end{figure}

In the next subsection we study properties of this family of cycles. 
We denote by $levelCalls(\kappa, \ell)$ the number of times that the $\kappa$-cycle routine is called in a single cycle on level $\ell$ when employing cycle counter $\kappa$, and $totalCalls(\kappa, n)$ denotes the total number of calls to the cycle routine in a single cycle of $n$ levels and cycle counter $\kappa$. In the proofs we commonly employ the well-known identities:
\begin{equation} \label{eq:BinomIdentity}
\binom{a-1}{b}+\binom{a-1}{b-1}=\binom{a}{b} ~~\mbox{and}~~ \sum_{j=0}^{a} \binom{a}{j} = 2^a \, .
\end{equation}

\subsection{Theoretical properties} \label{subsec:TheoreticalProperties}

We begin by computing the number of recursive calls at each level of the $\kappa$-cycle. As we will see in \cref{cor:levelCalls}, the number of calls per level is a monotonically increasing function of $\kappa$ for $1\leq\kappa\leq \ell$. For $\kappa\geq \ell$, the number of calls per level remains constant, because the $\kappa$-cycle is equivalent to a $W$-cycle in this regime.

\begin{proposition}\label{prop:levelCalls}
The number of calls to the cycle routine on level $\ell$ of a single $\kappa$-cycle with $\kappa \geq 1$ is given by
$$
levelCalls(\kappa, \ell)=\sum_{j=0}^{\min(\kappa-1,\ell-1)}\binom{\ell-1}{j} \, .
$$
\end{proposition}

\begin{proof}
Note first that for $\kappa = 0$ this formula yields $levelCalls(0,\ell) \equiv 0$, as it should, because the cycle routine is never called recursively with a non-positive cycle counter in Algorithm \ref{alg:kappa_cycle}. Next, we apply induction over the levels.

For $\ell=1$:
$$
levelCalls(\kappa, \ell = 1) = \sum_{j=0}^{\min(\kappa-1,\ell-1)} \binom{\ell-1}{j} = \sum_{j=0}^{0} \binom{0}{j} = \binom{0}{0} = 1 \, ,
$$

\noindent which is correct, as there is always one call on the first level of a single cycle. Now, we assume by induction that this claim is true for level $\ell=p$ and prove it for $\ell=p+1$. To this end, we need to relate the number of calls on level $p+1$ to the number of calls on level $p$. Observe that in the processing on the finest level $1$ in Algorithm \ref{alg:kappa_cycle} we make two recursive calls to $\kappa$-cycles on level $2$: one with cycle counter $\kappa$ and one with cycle counter $\kappa - 1$. Observe also that level $p+1$ of the original cycle is level $p$ of the two cycles beginning on level $2$. This implies
$$
levelCalls(\kappa, p+1) = levelCalls(\kappa, p) + levelCalls(\kappa-1, p) \, .
$$

\noindent
Therefore, it remains to be proved that
$$
\sum_{j=0}^{\min(\kappa-1,p)}\binom{p}{j} = \sum_{j=0}^{\min(\kappa-1,p-1)}\binom{p-1}{j} + \sum_{j=0}^{\min(\kappa-2,p-1)}\binom{p-1}{j}\, .
$$

\noindent
We distinguish between two cases:
\begin{enumerate}
\item $\kappa \leq p$.

It follows that $\kappa-2 < \kappa-1 \leq p-1$, and therefore
\begin{eqnarray*}
\sum_{j=0}^{\min(\kappa-1,p-1)} \binom{p-1}{j} + \sum_{j=0}^{\min(\kappa-2,p-1)} \binom{p-1}{j} & = & \sum_{j=0}^{\kappa-1}\binom{p-1}{j} + \sum_{j=0}^{\kappa-2}\binom{p-1}{j} \\
& = & \binom{p-1}{0} + \sum_{j=1}^{\kappa-1} \left( \binom{p-1}{j} + \binom{p-1}{j-1} \right) \\ & = & \binom{p-1}{0} + \sum_{j=1}^{\kappa-1} \binom{p}{j} = \sum_{j=0}^{\kappa-1}\binom{p}{j} \\ & = & \sum_{j=0}^{\min(\kappa-1,p)} \binom{p}{j} \, .
\end{eqnarray*}

\item $\kappa\geq p+1$.

It follows that $\kappa-1 > \kappa-2 \geq p-1$, and therefore
\begin{eqnarray*}
\sum_{j=0}^{\min(\kappa-1,p-1)} \binom{p-1}{j} + \sum_{j=0}^{\min(\kappa-2,p-1)} \binom{p-1}{j} & = & \sum_{j=0}^{p-1} \binom{p-1}{j} + \sum_{j=0}^{p-1} \binom{p-1}{j} \\ & = & 2^{p-1} + 2^{p-1} = 2^p = \sum_{j=0}^{p}\binom{p}{j} \, .
\end{eqnarray*}

\noindent
But $\kappa \geq p+1$ implies that $\min(\kappa-1, p)=p$, hence,
$$
\sum_{j=0}^{p} \binom{p}{j} = \sum_{j=0}^{\min(\kappa-1,p)} \binom{p}{j} \, .
$$
\end{enumerate}

\noindent
We find that in both cases we obtain $levelCalls(\kappa, p+1)=\sum_{j=0}^{\min(\kappa-1,p)}\binom{p}{j}$.
\end{proof}

\begin{corollary} \label{cor:levelCalls} 
\mbox{ }
\begin{enumerate}
\item
For $\ell > \kappa$ we have:
$$
levelCalls(\kappa,\ell)=\sum_{j=0}^{\min(\kappa-1,\ell-1)}\binom{\ell-1}{j}=\sum_{j=0}^{\kappa-1}\binom{\ell-1}{j} \, .
$$

\noindent
Hence, the number of calls per level grows monotonically with the cycle counter for $\ell > \kappa$.
\item
For $\ell \leq \kappa$ we have:
$$
levelCalls(\kappa,\ell)=\sum_{j=0}^{\min(\kappa-1,\ell-1)}\binom{\ell-1}{j}=\sum_{j=0}^{\ell-1}\binom{\ell-1}{j}=2^{\ell-1} \, .
$$

\noindent
Hence, the number of calls per level is independent of the cycle counter for $\ell \leq \kappa$.
\end{enumerate}
\end{corollary}

We next sum up the number of recursive calls over all the levels.

\begin{proposition} \label{prop:totalCalls}
The total number of calls to the cycle routine in a single $\kappa$-cycle with $\kappa \geq 1$ is given by
$$
totalCalls(\kappa, n) = \sum_{j=1}^{\min(\kappa,n)} \binom{n}{j} \, .
$$
\end{proposition}

\begin{proof}
We apply induction over the number of levels $n$. In the case of a single level, $n=1$, we have just one cycle call, and indeed
$$
totalCalls(\kappa,n=1) = \sum_{j=1}^{\min(\kappa,n)} \binom{n}{j} = \sum_{j=1}^{1}\binom{1}{j}=\binom{1}{1}=1 \, .
$$

\noindent
We assume that the claim is true for $n = p$ and prove it by induction for $n = p+1$.
Note that $totalCalls$ is nothing but the sum of $levelCalls$ over all levels, hence,
$$
totalCalls(\kappa,p+1) = totalCalls(\kappa,p) + levelCalls(\kappa,p+1) \, .
$$

\noindent
By the induction hypothesis and \cref{prop:levelCalls}, this yields
$$
totalCalls(\kappa,p+1) = \sum_{j=
1}^{\min(\kappa,p)}\binom{p}{j} + \sum_{j=0}^{\min(\kappa
-1,p)}\binom{p}{j} \, .
$$
\noindent
We distinguish between two cases:

\begin{enumerate}
\item
$\kappa \leq p$. Hence,

\begin{eqnarray*}
totalCalls(\kappa,p+1) & = & \sum_{j=1}^{\min(\kappa,p)} \binom{p}{j} + \sum_{j=0}^{\min(\kappa-1,p)} \binom{p}{j} = \sum_{j=1}^{\kappa} \binom{p}{j} + \sum_{j=0}^{\kappa-1} \binom{p}{j} \\ & = & \sum_{j=1}^{\kappa} \binom{p}{j} + \sum_{j=1}^{\kappa} \binom{p}{j-1} = \sum_{j=1}^{\kappa} \left( \binom{p}{j}+\binom{p}{j-1} \right) \\
& = & \sum_{j=1}^{\kappa} \binom{p+1}{j} = \sum_{j=1}^{\min(\kappa,p+1)} \binom{p+1}{j}
\end{eqnarray*}

\item
$\kappa \geq p+1$. Hence,

\begin{eqnarray*}
totalCalls(\kappa,p+1)& = & \sum_{j=1}^{\min(\kappa,p)} \binom{p}{j} + \sum_{j=0}^{\min(\kappa-1,p)} \binom{p}{j}   =  \sum_{j=1}^{p} \binom{p}{j} + \sum_{j=0}^{p} \binom{p}{j} \\ & = & 2^p-1+2^p=2^{p+1}-1 = \sum_{j=1}^{p+1} \binom{p+1}{j} = \sum_{j=1}^{\min(\kappa,p+1)} \binom{p+1}{j} \, .
\end{eqnarray*}
\end{enumerate}

\noindent
In both cases the claim is satisfied.
\end{proof}

\begin{corollary} \label{cor:CycleCallComplexity}
\mbox{ }
\begin{enumerate}
\item
For $\kappa \geq n$ we have $totalCalls(\kappa, n)=\sum_{j=1}^{n}\binom{n}{j}=2^n-1$, i.e., exponential in $n$ (as we observed in \cref{prop:VFW}).
\item
For $\kappa < n$ we have $totalCalls(\kappa, n) = \sum_{j=1}^{\kappa}\binom{n}{j}$. We conclude that when the number of levels $n$ is greater than the cycle counter $\kappa$, $totalCalls(\kappa, n)$ is a $\kappa$-degree polynomial in the number of levels $n$, because
$$
\binom{n}{j} = \frac{n!}{j!(n-j)!} = \frac{1}{j!} \left( n(n-1)(n-2) \ldots (n-j+1) \right)\, ,
$$

\noindent
a $j$-degree polynomial in $n$. Hence,
$$
totalCalls(\kappa,n) = \sum_{j=1}^{\kappa}\binom{n}{j} \, ,
$$

\noindent
is a polynomial in $n$ whose degree is $\kappa$.
\end{enumerate}
\end{corollary}

\noindent Evidently, the number of calls, and correspondingly the computational cost of the cycle, grows as $\kappa$ is increased until $\kappa = n$. On the other hand, as we shall demonstrate in our numerical tests, the convergence rate generally improves as $\kappa$ is increased. Our aim will therefore be to choose $\kappa$ that yields the best trade-off in terms of overall run time.  

\Cref{cor:CycleCallComplexity} leads us to the next formal observation on the linear complexity of the $\kappa$-cycle under suitable assumptions.

\begin{proposition} \label{prop:LinearComplexity}
Assume that the number of operations on the first (finest) level of the $\kappa$-cycle is linear in the number of variables $N$, that is, bounded from above by $CN$ as $N \rightarrow \infty$ for some constant $C$. Assume also that the number of operations per level is a monotonically decreasing function of the level $\ell$, and in fact the number of operations on any level $\ell$ is bounded from above by $c$ times the number of operations on the next-finer level $\ell-1$, for $\ell = 2, \ldots , n$, where $c < 1$ is a constant. Finally, assume that the number of operations per call to the coarsest level is a constant. Then, for any fixed cycle counter $\kappa$, the total number of operations per $\kappa$-cycle is $O(N)$ as $N \rightarrow \infty$. 
\end{proposition}

Note that this is in contrast to the $W$-cycle where, for example, even for $c = 0.5$ the total number of operations may be as high as $CN \log_2 N$. In particular, this occurs when the $W$-cycle is applied in geometric multigrid employing semi-coarsening for 2D problems.

\begin{proof}[Sketch of Proof]
The total number of operations in a single cycle is bounded by
$$
totalOps \leq CN \left( \sum_{\ell = 1}^n c^{\ell-1}levelCalls(\kappa,\ell)\right) \, .
$$

\noindent
Assuming $n > \kappa$ (else the result is obvious), we obtain by \cref{prop:levelCalls}
\begin{eqnarray*}
totalOps & \leq & CN \left(\sum_{\ell = 1}^{\kappa} c^{\ell-1} \sum_{j = 0}^{\ell-1} \binom{\ell-1}{j} + \sum_{\ell = \kappa + 1}^{n} c^{\ell-1} \sum_{j = 0}^{\kappa-1} \binom{\ell-1}{j}\right) \\ & \leq & CN \left( \sum_{\ell = 1}^{\kappa} (2c)^{\ell-1} + \sum_{\ell = \kappa + 1}^{\infty} c^{\ell-1} Poly_{\kappa-1}(\ell) \right)  \, ,
\end{eqnarray*}

\noindent where $Poly_{\kappa-1}$ is a polynomial of degree $\kappa - 1$. Observe that both terms in the final brackets are equal to some constants independent of $N$ or $n$, and therefore $totalOps$ is indeed linear in $N$. The final term can be bounded, e.g., by dominating it with an appropriate integral and performing multiple integration by parts. We omit the remaining details.
\end{proof}

\subsection{Number of operations in a geometric multigrid \texorpdfstring{{$\kappa$}}{k}-cycle}

\cref{prop:LinearComplexity} is mainly of academic interest, as the undetermined constants might be very large in some cases. In this subsection we focus on the practical determination of the number of operations performed in a $\kappa$-cycle for the case of geometric multigrid, where the coarsening factor is (exactly or nearly) a level-independent constant. This result will provide us with a practical tool in the next section. Because we will be using induction over levels, starting from the coarsest, in this subsection we reverse our convention and number the levels 1 to $n$ from coarsest to finest.  

Consider a $\kappa$-cycle with $n$ levels, where the number of variables on the coarsest level is a constant $N_1$  and the coarsening factor is a constant $c\in(0,1)$. Thus, the number of variables on the second-coarsest level is $N_2=c^{-1}N_1$,  and so on until the finest level where the number of variables is $N_n=c^{1-n}N_1$. Assume further that the number of operations performed at any level but the coarsest when the cycle routine is called is $C$ times the number of variables on that level, that is, $CN_n$ operations on the finest level, $CN_{n-1}=cCN_n$ on the second finest level, and so on until the second coarsest level, whereas on the coarsest level the problem is solved directly at a cost of $\tilde{C}N_1$ operations for some given constant $\tilde{C}$ which may be different from $C$. Then, the following proposition provides us with an exact formula for the total number of operations in a single $\kappa$-cycle.

\begin{proposition}\label{prop:n_ops}
The total number of operations in a single $\kappa$-cycle with $n$  levels is given by

\begin{equation} \label{eq:n_ops}
N^n_{ops}(\kappa, c)=f(\kappa, c)CN_n+P_{\kappa,c}(n) \, ,
\end{equation}
where
\begin{equation} \label{eq:n_ops_f}
f(\kappa, c)=\left\{ 
\begin{split}
& \frac{1}{1-2c}\left[1-\left(\frac{c}{1-c}\right)^\kappa\right], \, & if \, c \neq 0.5, \\
& 2\kappa , \, & if \, c=0.5, 
\end{split}
\right.  
\end{equation}
with (corresponding to a $W$-cycle) $f(\infty, c)=\frac{1}{1-2c}$ if $c<0.5$, while for $c\geq0.5$, $f(\infty, c)$ is undefined and $N^n_{ops}(\infty, c)$ has superlinear complexity. In (\ref{eq:n_ops}),
\begin{eqnarray*}
P_{\kappa,c}(n) & = & \sum_{j=0}^{min(\kappa-1,n-1)}[\tilde{C}-f(\kappa-j,c)C]N_1\binom{n-1}{j} \\
& \leq & \sum_{j=0}^{min(\kappa-1,n-1)} \left[ \tilde{C}-f(1,c)C \right] N_1\binom{n-1}{j} \\ & = & \left(\tilde{C}-\frac{C}{1-c}\right)N_1\sum_{j=0}^{min(\kappa-1,n-1)}\binom{n-1}{j}.
\end{eqnarray*}
\end{proposition}

Note that $P_{\kappa,c}(n)$ is not necessarily positive. Note also that for $0<c<0.5$ we have $\frac{1}{1-c}=f(1,c) \leq f(\kappa,c) \leq f(\infty,c)=\frac{1}{1-2c}$ so $P_{\kappa,c}(n)$ is bounded from above and below as follows.
$$
\left(\tilde{C}-\frac{C}{1-2c}\right)N_1\sum_{j=0}^{min(\kappa-1,n-1)}\binom{n-1}{j}\leq P_{\kappa,c}(n)\leq\left(\tilde{C}-\frac{C}{1-c}\right)N_1\sum_{j=0}^{min(\kappa-1,n-1)}\binom{n-1}{j}.
$$

We prove this proposition with the aid of two lemmas. The first is a simplification of \cref{prop:n_ops}, where the cost of the coarsest level solution is modified in a way that leaves only the linear term in (\ref{eq:n_ops}). The second lemma yields the correction required for the actual cost of the coarsest level solution in \cref{prop:n_ops}.

\begin{lemma} \label{prop:simple_n_ops}
Consider the $\kappa$-cycle of \cref{prop:n_ops}, modified such that the cost of the coarsest level solution is given by $f(\kappa, c)CN_1$ instead of $\tilde{C}N_1$, where the argument $\kappa$ is the cycle counter appearing in the call to the routine on the coarsest level. Then, the claim of \cref{prop:n_ops} holds with the second term eliminated, i.e., 
\begin{equation} \label{eq:simple_n_ops}
N^n_{ops}(\kappa, c)=f(\kappa, c)CN_n \, ,
\end{equation}
with $f(\kappa, c$) as defined in (\ref{eq:n_ops_f}). 
\end{lemma}

\noindent Note that the assumption on the modified cost of the coarsest level solve implies that it is not constant but rather varies within the $\kappa$-cycle.

\begin{proof}
We employ induction over the cycle counter $\kappa$, with the induction step itself proved by an inner induction over the number of levels in the cycle. Note that the modified cost of the coarsest level solution ensures that (\ref{eq:simple_n_ops}) is automatically satisfied on the coarsest level for any $\kappa$. Therefore, when executing the inner induction over levels it only remains to prove the induction step.

\noindent
$\kappa=1$ \emph{inner induction step}: assume that the induction hypothesis holds for a 1-cycle (i.e., $\kappa$-cycle with $\kappa=1$), with $n\geq1$ levels, that is, $N^n_{ops}(1,c)=f(1,c)CN_n=\frac{1}{1-c}CN_n$. Then, by the definition of the $\kappa$-cycle, the induction hypothesis, and the given constant coarsening factor $N_n=cN_{n+1}$, the number of operations in a 1-cycle of $n+1$ levels satisfies 
\begin{eqnarray*}
N^{n+1}_{ops}(1, c) & = & CN_{n+1}+N^{n}_{ops}(1,c)
 =  CN_{n+1}+f(1,c)CN_n \\
& = & CN_{n+1} \left( 1+c\frac{1}{1-c} \right)
 =  \frac{1}{1-c}CN_{n+1} = f(1,c)CN_{n+1} \, .
\end{eqnarray*}
We conclude that the claim holds for $\kappa=1$ and any number of levels.

\noindent
\emph{Outer induction step}: assume that the induction hypothesis holds for a $\kappa$-cycle with given $\kappa \geq 1$ and any number of levels, that is, $N^{n}_{ops}(\kappa,c)=f(\kappa,c)CN_{n}$. For the  ($\kappa+1$)-cycle with $n=1$ the claim holds automatically as noted above, by the assumption on the modified cost of the coarsest level solution. Assume then that the claim holds for the ($\kappa+1$)-cycle and $n\geq1$ levels. Then, by the definition of the $\kappa$-cycle, the induction hypothesis, and the given constant coarsening factor $N_n=cN_{n+1}$, the number of operations in a ($\kappa+1$)-cycle of $n+1$ levels satisfies for $c \neq 0.5$,
\begin{eqnarray*}
N^{n+1}_{ops}(\kappa+1, c) & = & CN_{n+1} + N^{n}_{ops} (\kappa+1,c) + N^{n}_{ops}(\kappa,c) \\
& = & CN_{n+1} + f(\kappa+1,c) CN_{n} + f(\kappa,c)CN_{n} \\
& = & CN_{n+1} \left(1+\frac{c}{1-2c} \left[1 - \left(\frac{c}{1-c} \right)^{\kappa+1} + 1 - \left( \frac{c}{1-c} \right)^\kappa \right] \right) \\
& = & CN_{n+1} \left( 1 + \frac{2c}{1-2c} - \frac{c}{1-2c} \left[ \left( \frac{c}{1-c} \right)^{\kappa+1} + \left( \frac{c}{1-c} \right)^\kappa \right] \right) \\
& = & CN_{n+1} \frac{1}{1-2c} \left(1-c \left( \frac{c}{1-c} \right)^{\kappa+1} \left( 1+\frac{1-c}{c} \right) \right) \\
& = & CN_{n+1} \frac{1}{1-2c} \left(1 - \left( \frac{c}{1-c} \right)^{\kappa+1} \right) = f(\kappa+1,c)CN_{n+1}.
\end{eqnarray*}

\noindent
For $c=0.5$ we have
\begin{eqnarray*}
N^{n+1}_{ops}(\kappa+1, c) & = & CN_{n+1} + N^{n}_{ops} (\kappa+1,c) + N^{n}_{ops}(\kappa,c) \\
& = & CN_{n+1} + 2(\kappa+1)CN_{n} + 2\kappa CN_{n} \\
& = & CN_{n+1} \left[1 + 0.5(4\kappa+2) \right] \\ & = & 2(\kappa+1) CN_{n+1} = f(\kappa+1,c)CN_{n+1},
\end{eqnarray*}
completing the proof.
\end{proof}

To complete the proof of \cref{prop:n_ops}, we need to subtract off the modified coarsest level costs of Lemma \ref{prop:simple_n_ops}, and add the corresponding coarsest level costs of the proposition. For this we need to derive how many times the routine is called at the coarsest level with each value of the cycle counter. This is done in the following lemma.

\begin{lemma} \label{lem:CoarseLevelCorrection} 
In a single complete $\kappa$-cycle with $n$ levels, the routine is called $\binom{n-1}{j}$ times at the coarsest level with cycle counter $\kappa-j$, for $j=0,...,min(\kappa-1,n-1)$.
\end{lemma}

\noindent
Note that the total number of calls at the coarsest level agrees with the result of Proposition \ref{prop:levelCalls}, but there we did not track the cycle counter at each call.

\begin{proof}[Sketch of Proof] 
For $\kappa=1$ the result is obvious, the routine is called just once on the coarsest level with cycle counter 1. For $\kappa>1$ we apply induction over the number of levels. For $n=2$ the routine is called on the coarsest level once with cycle counter $\kappa$ and once $\kappa-1$, corresponding to $j=0,1$, respectively, so the claim is satisfied. Assume now that the claim is satisfied for a given $\kappa\geq2$ and $n\geq2$ levels. By the definition of the $\kappa$-cycle, at level $n+1$ we call the routine recursively at level $n$, once with cycle counter $\kappa$ and once with $\kappa-1$. The result follows from the induction hypothesis and the fact that for $j>0$ we have
$$
\binom{n}{j}=\binom{n-1}{j}+\binom{n-1}{j-1},
$$
whereas for $j=0$
$$
\binom{n}{j}=\binom{n-1}{j}=1,
$$
leading to the stated result.
\end{proof}

These two lemmas prove \cref{prop:n_ops}, with the final inequality in this proposition resulting from the fact that $f(\kappa,c)$ is positive and monotonically increasing for any positive $\kappa$ and $c \in (0,1)$.

We remark, finally, that in our tests in \cref{sec:timePrediction} $c=0.25$, and the coarsest level solution is very cheap, so $P_{\kappa,c}$ is negligible compared to $f(\kappa,c)CN_n$ except for very small problems. Therefore, the formula of Lemma \ref{prop:simple_n_ops} provides us with an accurate value for $N_{ops}$.

\section{Predicting \texorpdfstring{{$\kappa$}}{k}-cycle run-time on GPU processors} \label{sec:timePrediction}

In this section we introduce a very simple model for predicting the approximate run-time of a single $\kappa$-cycle on GPU processors. When using GPUs for the computation, the CPU launches GPU-specific functions for execution on the GPU. These are called GPU kernels. As discussed in \cite{DBLP:conf/ics/VernerSS11}, the throughput of GPU kernels is in general not linear with the input size, and the accurate prediction of a GPU kernel runtime is quite complicated. However, we show that in our particular case the run-time per cycle can easily be calculated in advance fairly accurately, and therefore this simple predictive tool may be useful in choosing the best cycle for a given problem on a given system.

The model bases the run-time prediction on two system-dependent factors, and thus also gives an indication of the relative efficiency of different systems with respect to $\kappa$-cycles. All computations in our tests are done on GPU processors. The comparison includes the standard $V$, $F$ and $W$-cycles as the special cases $\kappa=1,2$ and $\kappa=\infty$, respectively. In addition, the cases of $\kappa=3$ and $\kappa=4$, which do not correspond to any previously known multigrid cycle, are included.

In our code, each relaxation, restriction, prolongation, etc., is implemented as a GPU kernel, and the number of GPU kernel calls can be calculated in advance. Before calling the cycle routines, the CPU prepares the problem. The data are then copied once from the CPU memory to the GPU memory, and after that all the data reside in the GPU memory and all the computations are done by the GPU.

\subsection{The model}

Here, a simple model for the run-time of a single cycle is presented. The model predicts the approximate run-time of one cycle, given the number of unknowns on the finest level and the cycle counter $\kappa$. The model assumes that the run-time for one cycle is a linear combination of the number of GPU kernel launches and the number of operations done by the kernels (which may include memory accesses, as well as arithmetic calculations), i.e.,  

\begin{equation} \label{eq:LinearCombination}
T=\alpha\cdot N_{gpuCalls}+\beta\cdot N_{ops}.
\end{equation}

\noindent
Here, $\alpha$ and $\beta$ depend on the system (specific CPU, GPU, OS, etc.), and on particular problem and cycle properties such as the discrete operator, coarsening factor, relaxation type, number of relaxations, prolongation and restriction operation, etc. They do not depend on the number of levels and on $\kappa$, as these dependencies are included in $N_{gpuCalls}$ and in $N_{ops}$. Once the parameters of the cycle are fixed, $\alpha$ and $\beta$ will have the same approximate values for any number of levels $n$ and any cycle counter $\kappa$.

$N_{gpuCalls}$ is the number of GPU kernel launches per cycle, which we can calculate. In our implementation, there is one GPU kernel launch in each routine call at the coarsest level, and a fixed number of GPU kernel launches in each routine call on the finer levels, which depends on the number of relaxations. Using this knowledge together with the $totalCalls(\kappa, n)$ value from \cref{sec:TheKappaCycle}, we can predict the total number of GPU kernel launches. $N_{ops}$ is a measure for the total amount of GPU operations, such as memory reads/writes and arithmetic operations. We show how to compute this value in our particular example in (\ref{eq:n_ops_2D}). A summary of all the formulas is given in \cref{table:notations}.

This model is based on a simplified assumption that each GPU kernel has an overhead which takes a constant time per launch, and that the time required for the kernel itself is in direct proportion to the amount of operations it does, dominated by memory reads and writes in our case, except for very small problems. Because the cycle routine on each level but the coarsest has the same GPU kernel launches (same number and same GPU functions), the total memory reads and writes are in direct proportion to the number of variables in that level, and therefore the time spent by GPU kernels is also proportional to this number. CPU time is not considered in the model, because the CPU execution is parallel to GPU execution and CPU time is much smaller. Note that in practice these operations are affected by caching, operation latencies and other considerations.

Below, we present results of numerical tests performed in order to assess this model, and we show that our model is able to usefully predict the actual run-time per cycle. We therefore conclude that our simplifications are reasonable. Note that our tests show that the bottleneck of our kernels are the memory transfers and not the computations themselves, but the amount of memory transferred and the amount of computations are both in direct proportion to the number of unknowns, so the same model would be valid either way.

In the tests, the problem of 2D rotated anisotropic diffusion with constant coefficients is solved. The problem is discretized on a square grid employing a nine-point finite-difference stencil. More details are provided in the next section. Only a stand-alone $\kappa$-cycle is considered in this section. We use standard coarsening, so the coarsening factor is approximately $c=0.25$. Two pre-relaxations and two post-relaxations are employed in this section, and we use a maximum of 13 levels, which implies $(2^{13}-1)^2 = 8191^2 = 67,108,864$ unknowns on the finest level. We use double precision floats, which are 8 bytes each. For each level there are 3 pre-allocated arrays: current estimation, right-hand side, and temporary values. 

\Cref{fig:cycle_times}(a) shows the mean measured times for one cycle with $\kappa=1,2,3,4$ and $\kappa=\infty$, for 4-13 levels. The vertical lines in the graph show where the operation times become equal to the overhead times, referred to as ``turning points'' below.

\subsection{Estimating the model parameters}

In order to approximate the constants $\alpha$ and $\beta$ in our system \eqref{eq:LinearCombination}, the simple least-squares problem \eqref{eq:MinProblem} is solved for $\alpha$ and $\beta$. 

\begin{equation} \label{eq:MinProblem}
(\alpha,\beta) = argmin_{\tilde{\alpha},\tilde{\beta}} \sum_{\kappa=1,2,3,4,\infty}\sum_{n=4}^{13}(\tilde{\alpha}\cdot N_{gpuCalls}(\kappa,n)+\tilde{\beta}\cdot N_{ops}(\kappa,n)-T(\kappa,n))^2,
\end{equation}

\noindent
where $T(\kappa,n)$ is the measured time for one $\kappa$-cycle (averaged over 200 runs), $N_{gpuCalls}$ is the number of GPU kernel launches in one cycle, and $N_{ops}(\kappa,n)$ is proportional to the total number of memory accesses. (The exact number depends on the number and type of the relaxations and other factors, which we assume to be constant and represented in $\beta$). They, in turn, are proportional to the number of unknowns. $N_{gpuCalls}$ and $N_{ops}$ are computed according to \cref{table:notations}. In our main test system, the values of $\alpha$ and $\beta$ were found to be $\alpha=2.48\cdot 10^{-3}$ ms, $\beta=1.18\cdot 10^{-6}$ ms. These two model parameters suffice to predict fairly accurate run-times for any $\kappa$ and problem size, as seen in \cref{fig:cycle_times}(b-f).

Actual mean run-times for a single cycle with $\kappa=1,2,3,4,\infty$, are shown in black in \cref{fig:cycle_times}(b-f). The times predicted by \eqref{eq:LinearCombination} are plotted in blue. All graphs use the values of $\alpha$ and $\beta$ as stated above. The red curves show the call cost component of \eqref{eq:LinearCombination}, while the green curve shows the computation cost component, hence, the blue curve is the sum of the red and the green curves. As the number of levels $n$ increases, the cost of operations grows faster than the cost of calls, so for larger $\kappa$'s the operation costs catch up with the overhead at larger problems (at the so-called turning point, where the green and red curves cross in \cref{fig:cycle_times}). This is due to the fact that larger $\kappa$ implies more calls at coarse levels.

\begin{figure}[hpbt!]
\centering
\subfigure[Measured times for $\kappa$-cycles on the GPU.]{
\includegraphics[scale=0.8]{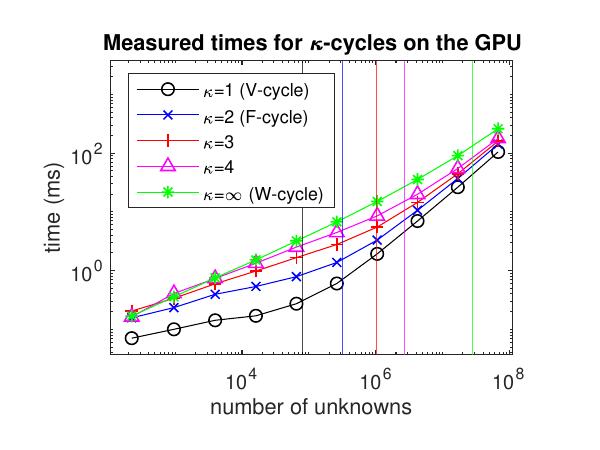}}
\subfigure[$V$-cycle times and predictions.] {\includegraphics[scale=0.8]{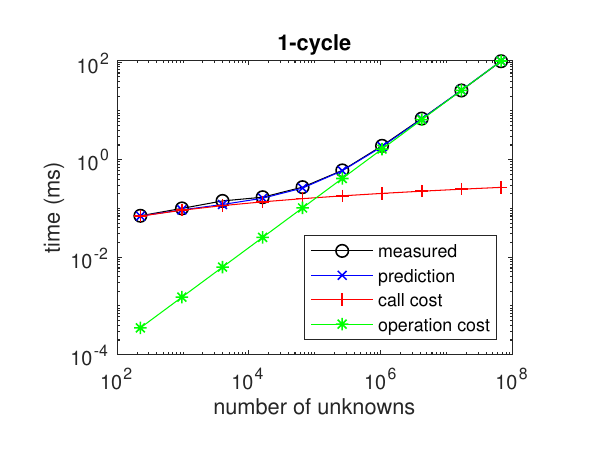}}
\subfigure[$F$-cycle times and predictions.] {\includegraphics[scale=0.8]{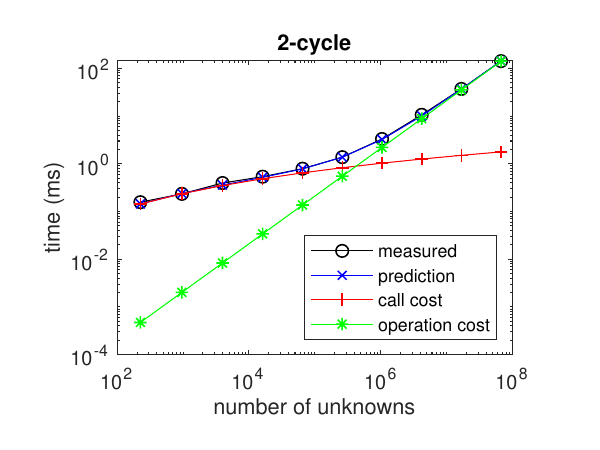}}
\subfigure[$\kappa$-cycle times and predictions for $\kappa=3$.] {\includegraphics[scale=0.8]{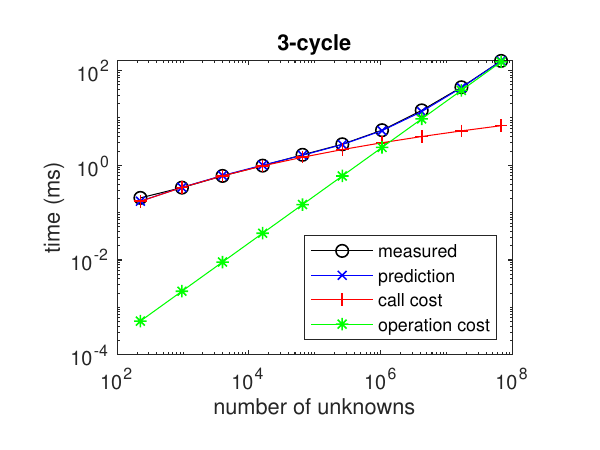}}
\subfigure[$\kappa$-cycle times and predictions for $\kappa=4$.] {\includegraphics[scale=0.8]{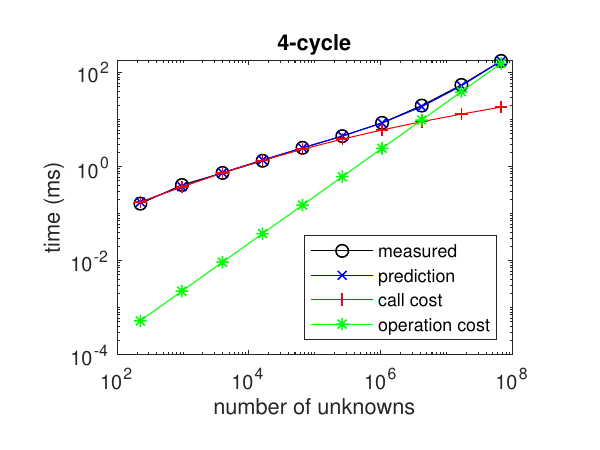}}
\subfigure[$W$-cycle times and predictions.] {\includegraphics[scale=0.8]{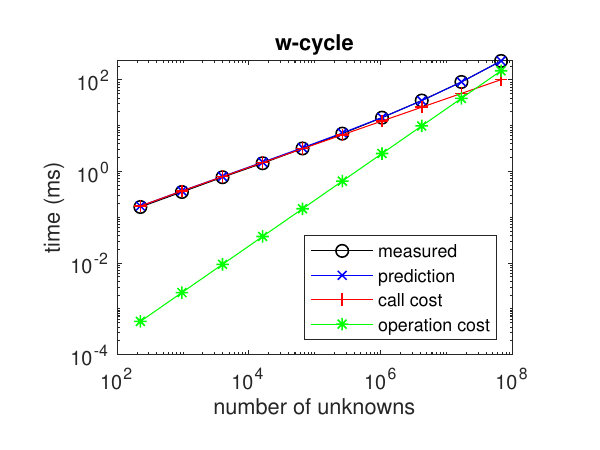}}
\caption{Measured and predicted run-times per cycle for $\kappa=1,2,3,4,\infty$. The vertical lines in panel (a) show the turning points, where red and green curves cross in panels (b)-(f).}
\label{fig:cycle_times}
\end{figure}

\subsection{Values for \texorpdfstring{{$N_{gpuCalls}$ and $N_{ops}$}}{NgpuCalls and Nops} in our implementation}
\label{sec:n_comp}

There are $5$+$\nu$ GPU kernel calls per cycle routine call, where $\nu=\nu_1+\nu_2$ is the total number of relaxations, except at the coarsest level, where there is one GPU kernel launch, as seen in Algorithm \ref{alg:launches}. The number of GPU kernel launches per cycle is calculated as shown in \cref{table:notations}.
\begin{algorithm2e} \label{alg:launches}
\DontPrintSemicolon
\label{alg:kappa_cycle_kernels}
\caption{Kernel launches in the $\kappa$-cycle}
{$v^\ell \leftarrow \kappa\mbox{-}cycle(v^\ell,f^\ell,A^\ell,\ell,n,\kappa)$}\;
\Indp
\nl
If $\ell == n$ (coarsest level), solve $A^\ell v^\ell = f^\ell$ and return. - \textbf{1 kernel launch on coarsest level}\;
\nl
{\em Relax} $\nu_1$ times on $A^\ell u^\ell = f^\ell$ with initial guess $v^\ell$. - \textbf{$\nu_1$ kernel launches}\;
\nl
$f^{\ell+1} \leftarrow Restrict(f^\ell - A^\ell v^\ell)$. - \textbf{$2$ kernel launches}: 1 residual + 1 restrict\;
\nl
$v^{\ell+1} \leftarrow 0$. - \textbf{$1$ kernel launches}: make the vector zero\;
\nl
$v^{\ell+1} \leftarrow \kappa\mbox{-}cycle(v^{\ell+1},f^{\ell+1},A^{\ell+1},\ell+1,n,\kappa)$.\;
\nl
If $\kappa > 1$\;
$v^{\ell+1} \leftarrow \kappa\mbox{-}cycle(v^{\ell+1},f^{\ell+1},A^{\ell+1},\ell+1,n,\kappa - 1)$.\;
\nl
$v^\ell \leftarrow v^\ell + Prolong(v^{\ell+1})$. - \textbf{$2$ kernel launches}: 1 prolong + 1 vector addition\;
\nl
{\em Relax} $\nu_2$ times on $A^\ell u^\ell = f^\ell$ with initial guess $v^\ell$. - \textbf{$\nu_2$ kernel launches}\;
\end{algorithm2e}

Regarding $N_{ops}$, the last level in our tests has only one unknown ($N_1=1$). Therefore, $P_{\kappa,c}(n)$ is negligible compared to $f(\kappa, c)CN_n$ in (\ref{eq:n_ops}), and we can use (\ref{eq:simple_n_ops}). Also, when using standard coarsening for 2D problems, the coarsening factor is very close to $c=0.25$. Using $c=0.25$ in (\ref{eq:simple_n_ops}) yields:

\begin{eqnarray}
\label{eq:n_ops_2D}
N^n_{ops}(\kappa, 0.25) & = & f(\kappa, 0.25)CN_n  = \\ & & \frac{1}{1-2\cdot0.25}\left[1-\left(\frac{0.25}{1-0.25}\right)^\kappa\right]CN_n = 2\left(1-\frac{1}{3^\kappa}\right)CN_n \, . \nonumber
\end{eqnarray}

\subsection{Calculating the turning point}

The problem size at the turning point, where the overhead time and computation time are equal, is denoted $N_{tp}$. This notion is relevant because the number of operations grows linearly with the number of unknowns, hence, much faster than the growth due to calls (compare the slopes of the red and green curves in \cref{fig:cycle_times}). The turning point, where the green and red curves cross each other, provides an indication of the approximate problem-size for which the $\kappa$-Cycle enjoys its maximal relative advantage in terms of run-time, compared to the cycle obtained when $\kappa$ is increased by one.

Evidently, the turning point does not depend on $\alpha$ and $\beta$ separately, but rather only on their ratio:

\begin{equation} \label{eq:equalCosts}
\alpha\cdot N_{gpuCalls}=\beta\cdot N_{ops} \implies N_{ops}=\frac{\alpha}{\beta}\cdot N_{gpuCalls}
\end{equation}

\noindent
For the values $\alpha=2.48\cdot 10^{-3}$, $\beta=1.18\cdot 10^{-6}$ of our system, $N_{tp}$ and corresponding (unrounded) finest levels for various $\kappa$'s are shown in \cref{table:equal_times}. Details on how we calculate these numbers are shown in the appendix.

\begin{table}
\centering
\caption{The approximate turning point and corresponding unrounded finest level}
\begin{tabular}{ |c|c|c| }
\hline
 $\kappa$ & $n_{tp}$ & $N_{tp}$ \\  
\hline
 1 ($V$-cycle) & 8.2 & 80,000 \\
\hline
 2 ($F$-cycle) & 9.1 & 320,000 \\
\hline
 3 & 10.0 & 1,000,000 \\
\hline
 4 & 10.7 & 2,700,000 \\
\hline
 $\infty$ ($W$-cycle) & 12.4 &  28,000,000 \\
\hline
\end{tabular} 
\label{table:equal_times}
\end{table} 

We find that when there are up to $10^7$ unknowns (12 levels or less), most of the run-time of a $W$-Cycle is spent on overhead. In a $V$-cycle, in contrast, about $10^5$ unknowns (9 levels) are sufficient for the calculation time to exceed the overhead time.

\subsection{Relative prediction errors} \cref{table:relative_errors} shows the relative errors of the model when compared to actual mean run-times. It can be seen that the relative prediction errors are just a few percent, and the prediction is especially accurate for large run-times. We explain the somewhat lower accuracy for shorter runs by noticing that for seven levels or less the required data fit entirely into the GPU cache: the GPU in our main system has L2 cache of 1.5MB, and the amount of memory required is about $3\cdot 8\cdot 4^{N_{levels}}$ bytes (about $4^{N_{levels}}$ doubles per array, 3 arrays per level, 8 bytes per double). This causes the execution of the kernels to become very fast, and the required time is dominated by the kernel overhead. The proposed model does not account for this phenomenon, but in big problems this has a small effect on the overall time.


\begin{table}
\centering
\caption{Relative run-time prediction errors}
\label{table:relative_errors}
\begin{tabular}{ |c|c|c|c|c|c|c| }
\hline
 $\kappa$ $\setminus$ n & 8 & 9 & 10 & 11 & 12 & 13 \\
\hline
1 &  -4.28\% & -1.68\% & -3.64\% & -1.96\% & 1.55\% & 1.42\% \\
\hline
2 & -0.80\% & 0.05\% & -2.19\% & -5.14\% & -0.50\% & 0.90\% \\
\hline
3 & -2.13\% & -1.35\% & -2.33\% & -5.94\% & -2.18\% & -0.12\% \\
\hline
4 & -0.49\% & 0.08\% & -0.72\% & -5.28\% & -4.38\% & -1.02\% \\
\hline
 $\infty$ & 3.61\% & 3.21\% & 1.07\% & -0.62\% & 0.12\% & 0.30\% \\
\hline
\end{tabular}
\end{table}

\begin{table}
\caption{Notation and formulas}
\begin{tabular}{ |c|c|c|}
\hline
 Symbol & Meaning & Formula  \\  
\hline
 $n$ / $N_{levels}$ & number of levels &  \\
\hline
 $N$ / $N_{unknowns}$ & number of unknowns on the finest level & $(2^{N_{levels}}-1)^2$\\
\hline
 $\kappa$ & {\em cycle counter} & \\
\hline
 $levelCalls(\kappa, \ell)$ & number of routine calls on level $\ell$ per cycle & $\sum_{j=0}^{\min(\kappa-1,\ell-1)}\binom{\ell-1}{j}$ \\
\hline
 $totalCalls(\kappa, n)$ & number of routine calls per cycle & $\sum_{j=1}^{\min(\kappa,n)} \binom{n}{j}$ \\
\hline
$C(\kappa)$ & see formula & $2\left(1-\frac{1}{3^\kappa}\right)$ \\
\hline
$\nu$ & number of relaxations per routine call & $\nu_1 + \nu_2$\\
\hline
$N_{gpuCalls}$ & number of GPU kernel & $(5+\nu)\cdot totalCalls(\kappa,n-1)+$ \\ 
 & launches per cycle & $totalCalls(\kappa, n)-totalCalls(\kappa, n-1)$ \\
\hline
 $N_{ops}$ & measure for the total amount of operations & $N_{unknowns}\cdot2\left(1-\frac{1}{3^\kappa}\right)$ \\
 & done by the GPU in one cycle (see formula) & \\
\hline
\end{tabular}
\label{table:notations}
\end{table}

\section{Numerical results} \label{sec:NumericalTests}

In this section we report some numerical tests and results in greater detail. 
The run-time per cycle is monotonically increasing with $\kappa$ (until $\kappa = n$, as from there on the $\kappa$-cycle coincides with the $W$-cycle). However, a larger $\kappa$ typically results in faster convergence per cycle. Thus, there is a tradeoff in selecting $\kappa$, and the optimal choice is problem-dependent and system-dependent. To demonstrate the potential of the $\kappa$-cycle, we test our algorithms on a 2D rotated anisotropic diffusion problem, on a square domain $\Omega$ and with Dirichlet boundary conditions:
\begin{align*}
-\epsilon u_{xx} - u_{yy} &= f(x,y) \, , \, (x, y) \in \Omega \, , \\
u &= g(x, y) \, , \, (x, y) \in \partial \Omega \, ,
\end{align*}

\noindent discretized on a square equi-spaced grid, where $0< \epsilon \leq 1$ is a parameter, and the coordinates $(x,y)$ form an angle $\phi$ with the grid-lines. 
In the tests we use a right-hand side and boundary conditions of zero, hence the exact solution is zero, and the starting guess is random (but it is the same in all tests). This allows us to check the asymptotic behavior of the method without numerical round-off errors (without loss of generality, because the problem is linear and the method is stationary).
This problem is challenging for standard geometric multigrid with simple coarsening when $\epsilon$ is small and the coordinates are not aligned with the grid, e.g., for $\phi = \pi/4$, which is our choice in most of the tests (see, e.g., \cite{Yav98}). We use the following nine-point discretization stencil: 

\medskip
\begin{center}
\begin{tabular}{ |c|c|c| }
\hline
 $\frac{1}{2}(1-\epsilon)CS$ & $-(\epsilon C^2+ S^2)$ & $-\frac{1}{2}(1-\epsilon)CS$ \\ 
\hline
 $-(C^2+\epsilon S^2)$ & $2(1+\epsilon)$ & $-(C^2+\epsilon S^2)$ \\  
\hline
 $-\frac{1}{2}(1-\epsilon)CS$ & $-(\epsilon C^2+ S^2)$ & $\frac{1}{2}(1-\epsilon)CS$ \\
\hline
\end{tabular} 
\end{center}
\medskip

\noindent
where $C=cos(\phi)$, $S=sin(\phi)$.

Six types of solvers are tested. 
One uses a stand-alone $\kappa$-cycle (which includes the $V$-cycle, $F$-cycle and $W$-cycle as special cases) with standard coarsening and simple relaxation---Jacobi with damping factor between 0.8 and 0.87 (following \cite{YO98}). A second solver uses Conjugate Gradients, preconditioned by a single $\kappa$-cycle per iteration, where the $\kappa$-cycle is the same as in the previous solver. 
The third solver uses an alternating ``zebra'' relaxation---see details below, and the standard full coarsening, together with Galerkin coarse-grid operators. 
The fourth solver uses Conjugate Gradients, preconditioned by a single $\kappa$-cycle per iteration, where the $\kappa$-cycle is the same as in the third solver. 
The fifth solver tested uses a line Gauss-Seidel relaxation (in a ``zebra'' order with no relaxation parameter) and semi-coarsening, with Galerkin coarse-grid operators---see details below. The sixth solver uses Conjugate Gradients, preconditioned by a single $\kappa$-cycle per iteration, where the $\kappa$-cycle is the same as in the fifth solver. 

We apply two pre-relaxations and two post-relaxations, because this gives better times than other configurations with equal number of pre-relaxations and post-relaxations (for the third and fourth solvers, two relaxations mean one relaxation in each direction). In this section, we use a maximum of 14 levels, which results in $(2^{14})^2 = 16384^2 = 268,435,456$ unknowns on the finest level.  The first two   computations are performed on a system with NVIDIA GTX 1060 GPU, having 1280 cores and a peak theoretical memory bandwidth of 192 GB/s. This would allow the GTX 1060 a throughput of 1280 arithmetic instructions per GPU clock cycle, if the instructions were for single precision floats. However, the throughput for double precision\footnote{Double precision is required for obtaining accurate solutions in large problems of this type.} floats is only 40 instructions per clock cycle on the GTX 1060 (see CUDA site, especially \cite{cuda_c_programming_guide} and \cite{pascal_tuning_guide}). Indeed, the throughput for doubles in most GPUs currently in use is much smaller than their float throughout. However, the impact of this in our implementation seems to be small, because in our tests the memory bandwidth is the bottleneck (so long as the required data size is above the 1.5MB L2 cache of the GPU). Still, the double precision floats, as their name implies, require double the memory bandwidth compared to single precision floats, and the time for the arithmetic instructions may still impact the overall time. We remark that the cycle time may be improved by using a hybrid program, as recommended by \cite{HPGMG_GPU}, whereby a few coarse levels of the cycle are run on the CPU, and finer levels are run on the GPU. Here, however, we use only the GPU for the computations, in order to obtain a clean demonstration of our results. This is also justified by the fact that, if current trends continue, the added value of using CPU is likely to greatly diminish over time. Indeed, the advantage of a hybrid program already seems to be small in our implementation, probably because the time needed for copying the data to CPU memory is already close to the time needed for running the coarser levels on the GPU.

\cref{fig:solve_times} shows the total run-times for reducing the $L_2$ error norms by a factor of $10^8$ for  $\epsilon \in \{10^{-5}, 10^{-4}, 10^{-3}, 10^{-2}, 10^{-1}, 1\}$, using 9 levels (``small'' problem with about 260K variables on the finest grid) and 13 levels (``large'' problem with about 67M finest-grid variables). Here, the $\kappa$-cycle is used as a stand-alone solver. For $\epsilon=1$ this is the simple Poisson problem, and the $V$-cycle yields the minimal run-time. As $\epsilon$ is decreased, the problem becomes more challenging, especially the large problem, and the $V$-cycle loses its efficiency because of the fast deterioration in its convergence factor. Still, the $W$-cycle, which has the best convergence factor, is expensive. For the small problem, the $F$-cycle becomes the most efficient for $\epsilon \leq 0.01$. For the large problem, however, larger $\kappa$ values are optimal once $\epsilon$ drops below 0.01, and $\kappa = 4$ becomes the optimal parameter for still smaller $\epsilon$ values.

\begin{figure}[hpbt!]
\centering
\includegraphics[scale=0.8]{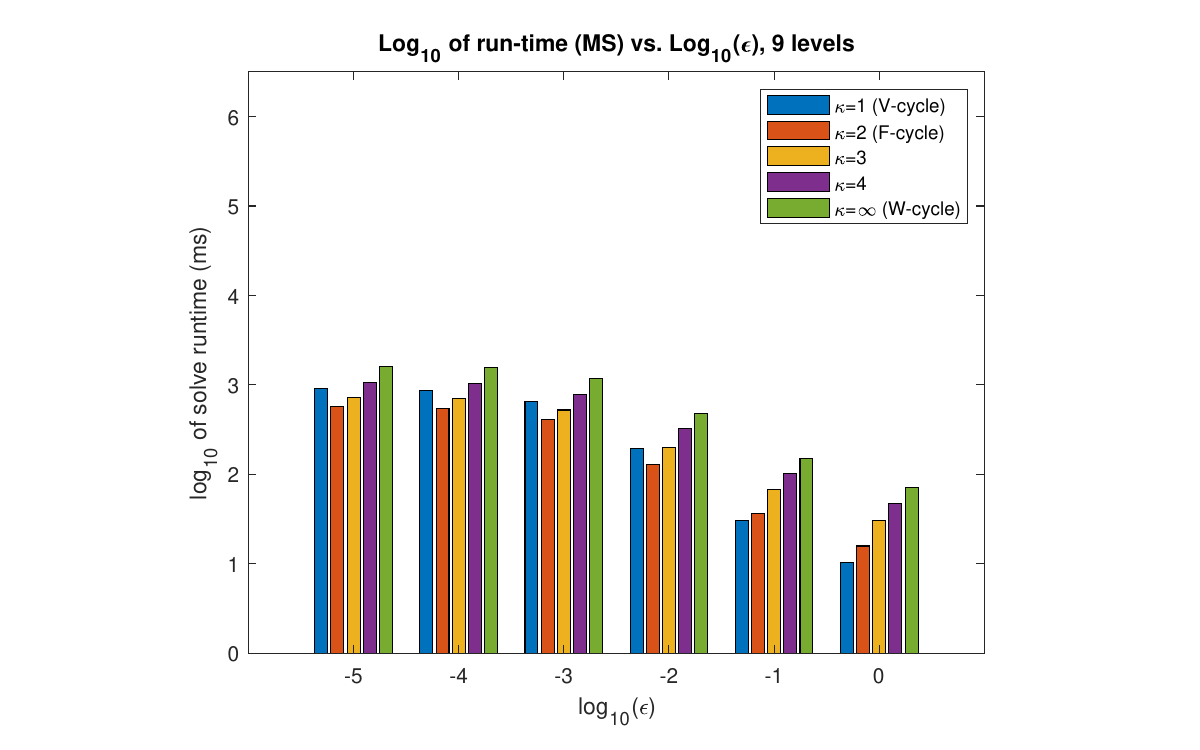}
\includegraphics[scale=0.8]{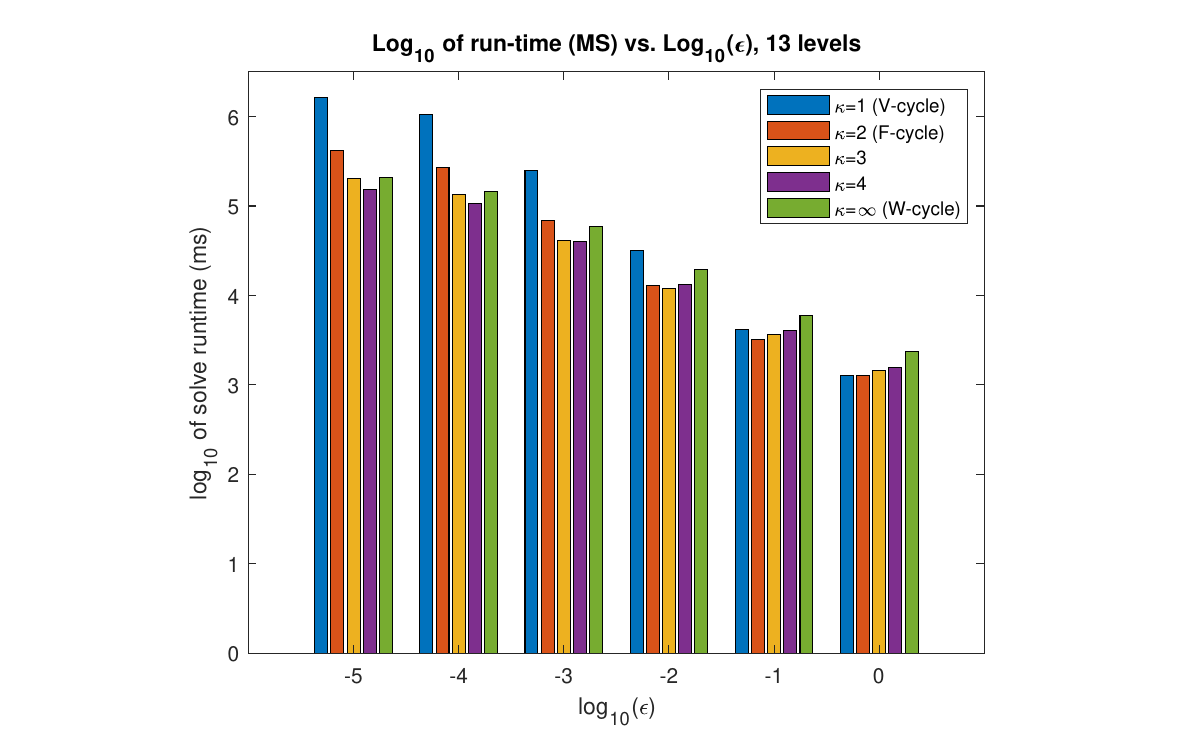}
\caption{Solve run-times}
\label{fig:solve_times}
\end{figure}

\begin{table}
\centering
\caption{Stand-alone multigrid with $\kappa(2,2)$, using damped Jacobi relaxation and standard coarsening, 12 levels, $\phi=45\degree$. Relative total time is the total time divided by that of the $W$-cycle}
\label{table:Stand_alone}
\begin{tabular}{ |c|c|c|c|c| }
\hline
 $\kappa$ & Total time(ms) & Relative total time & Average time per cycle (ms) & \#cycles \\  
\hline
 1 ($V$-cycle) & 182,216 & 4.29 & 26.4 & 6909 \\
\hline
 2 ($F$-cycle) & 52,080 & 1.23 & 37.1 & 1403 \\
\hline
 3 & 28,863 & 0.679 & 44.3 & 651 \\
\hline
 4 & \textbf{26,356} & \textbf{0.620} & 53.2 & 495 \\
\hline
 $\infty$ ($W$-cycle) & 42,483 & 1.00 & 90.4 & 470 \\
\hline
\end{tabular} 
\end{table}

\begin{table}
\centering
\caption{Conjugate Gradients preconditioned by a $\kappa$-cycle (MGCG), using damped Jacobi relaxation and standard coarsening, 12 levels, $\phi=45\degree$. Relative total time is the total time divided by that of the $F$-cycle}
\label{table:mgcg}
\begin{tabular}{ |c|c|c|c|c| }
\hline
 $\kappa$ & Total time(ms) & Relative total time & Average time per cycle (ms) & \#iterations \\  
\hline
 1 ($V$-cycle) & 9014 & 1.71 & 47.7 & 189 \\
\hline
 2 ($F$-cycle) & 5262 & 1.00 & 59.1 & 89 \\
\hline
 3 & \textbf{4252} & \textbf{0.808} & 67.5 & 63 \\
\hline
 4 & 4294 & 0.816 & 76.7 & 56 \\
\hline
 $\infty$ ($W$-cycle) & 6142 & 1.17 & 113.7 & 54 \\
\hline
\end{tabular}
\end{table}

\begin{table}
\centering
\caption[Stand-alone $\kappa$-cycle - zebra relaxations]{Stand-alone $\kappa$-cycle run-times (ms) using alternating zebra relaxation and Galerkin coarse-grid operators, 14 levels, the numbers in parentheses indicate the number of cycles}
\label{table:xyzebra_14_levels_e8}
\begin{tabular}{ |c|c|c|c|c|c| }
\hline
 $\kappa$ $\setminus$ angle & 10 & 30 & 45 & 60 & 80 \\
\hline
 1 ($V$-cycle) & 751452 (2909) & timeout & timeout & timeout & 835140 (3234) \\
\hline
 2 ($F$-cycle) & 197204 (589) & 451086 (1347) & 577985 (1725) & 553166 (1651) & 219901 (656) \\
\hline
 3 & 103595 (274) & 212531 (562) & 266935 (706) & 253346 (670) & 114530 (303) \\
\hline
 4 &\textbf{88675} (206) & \textbf{163207} (379) & \textbf{199501} (463) & \textbf{191667} (445) & \textbf{97792} (227) \\
\hline
 $\infty$ ($W$-cycle) & 159587 (193) & 268824 (326) & 316618 (385) & 313178 (381) & 174642 (212) \\
\hline
\end{tabular}
\end{table}

\begin{table}
\centering
\caption[Conjugate Gradients preconditioned by a $\kappa$-cycle - zebra relaxations]{Conjugate Gradients preconditioned by a $\kappa$-cycle (MGCG) run-times (ms) using alternating zebra relaxation and Galerkin coarse-grid operators, 14 levels, the numbers in parentheses indicate the number of iterations}
\label{table:xyzebra_14_levels_cg_e8}
\begin{tabular}{ |c|c|c|c|c|c| }
\hline
 $\kappa$ $\setminus$ angle & 10 & 30 & 45 & 60 & 80 \\
\hline
 1 ($V$-cycle) & 39348 (124) & 61707 (195) & 72305 (229) & 68893 (218) & 41898 (132) \\
\hline
 2 ($F$-cycle) & 23199 (58) & 33852 (85) & 38888 (98) & 37598 (95) & 24343 (61) \\
\hline
 3 & 18384 (41) & 25352 (57) & 28401 (64) & 27459 (62) & 19209 (43) \\
\hline
 4 & \textbf{18117} (36) & \textbf{23970} (48) & \textbf{26439} (53) & \textbf{25869} (52) & \textbf{19057} (38) \\
\hline
 $\infty$ ($W$-cycle) & 31548 (35) & 40367 (45) & 43921 (49) & 43763 (49) & 33166 (37) \\
\hline
\end{tabular}
\end{table}

\begin{table}
\centering
\caption[Stand-alone multigrid with $\kappa(2,2)$ - zebra relaxation, semi-coarsening]{Stand-alone multigrid with $\kappa(2,2)$ run-times (ms), using zebra relaxation, semi-coarsening and Galerkin coarse-grid operators, 14 levels, the numbers in parentheses indicate the number of cycles}
\label{table:xzebra_14_levels_1e8}
\begin{tabular}{ |c|c|c|c|c|c| }
\hline
 $\kappa$ $\setminus$ angle & 10 & 30 & 45 & 60 & 80 \\
\hline
 1 ($V$-cycle) &  618154 (1705) & timeout & timeout & timeout & 620274 (1711) \\
\hline
 2 ($F$-cycle) & 297372 (371) & 775767 (966) & 875116 (1093) & 677832 (846) & 193031 (241) \\
\hline
 3 & \textbf{281482} (187) & \textbf{594359} (395) & \textbf{616255} (410) & \textbf{450079} (299) & 114353 (76) \\
\hline
 4 & 410784 (148) & 730504 (263) & 704672 (254) & 482585 (174) & \textbf{110864} (40) \\
\hline
 $\infty$ ($W$-cycle) & timeout & timeout & timeout & timeout & 363991 (26) \\
\hline
\end{tabular}
\end{table}

\begin{table}
\centering
\caption[Conjugate Gradients preconditioned by a $\kappa$-cycle (MGCG) - zebra relaxation, semi-coarsening]{Conjugate Gradients preconditioned by a $\kappa$-cycle (MGCG) run-times (ms) with $\kappa(2,2)$, using zebra relaxation, semi-coarsening and Galerkin coarse-grid operators, 14 levels, the numbers in parentheses indicate the number of iterations}
\label{table:xzebra_14_levels_cg_1e8}
\begin{tabular}{ |c|c|c|c|c|c| }
\hline
 $\kappa$ $\setminus$ angle & 10 & 30 & 45 & 60 & 80 \\
\hline
 1 ($V$-cycle) & 39548 (93) & 70761 (164) & 79566 (189) & 69379 (165) & 39811 (94) \\
\hline
 2 ($F$-cycle) & \textbf{39377} (45) & \textbf{61557} (71) & \textbf{68511} (79) & \textbf{58007} (67) & \textbf{31492} (36) \\
\hline
 3 & 52292 (33) & 77894 (48) & 79088 (50) & 64627 (41) & 31889 (20) \\
\hline
 4 & 85674 (30) & 115353 (40) & 114044 (40) & 94055 (33) & 43082 (15) \\
\hline
 $\infty$ ($W$-cycle) & 411286 (29) & 538387 (38) & 510449 (36) & 408226 (29) & 170028 (12) \\
\hline
\end{tabular}
\end{table}

\Cref{table:Stand_alone} shows the total run-times for reducing the $L_2$ error norms by a factor of $10^8$ for $\epsilon=0.0001$ using 12 levels (about 17M finest-grid variables), where the $\kappa$-cycle is used as a stand-alone solver. These parameters result in a challenging problem, and we clearly see the tradeoff here, with the required number of cycles decreasing as $\kappa$ increases, but the average time per cycle increasing significantly. In terms of total run-time, the best tradeoff is provided for $\kappa=4$, with $\kappa=3$ close behind.

\Cref{table:mgcg} shows run-times for the case where the $\kappa$-cycle is used as a preconditioner for Conjugate Gradients (called MGCG in \cite{Tatebe93}). As expected, this yields a much more efficient solver than the stand-alone $\kappa$-cycle. We see once again that setting $\kappa$ to 3 or 4 yields superior run-times, this time with $\kappa = 3$ in the lead.


For the angle $\phi=45\degree$, the point relaxation used in the two tests described above is adequate, because the bottleneck is the coarse-grid correction. However, when $\epsilon$ is small and $\phi$ is close to 0 or 90 degrees, and thus the direction of strong coupling is nearly aligned with the grid, point relaxation such as Jacobi is known to be ineffective. In order to obtain a sufficiently good smoothing factor in these cases, there are two common methods: using line-relaxation along the coordinate of the strongly-coupled variables, or coarsening only in the coordinate of the strongly-coupled variables (semi-coarsening). Because in real problems we usually do not know in advance the direction of strong coupling (and it may change over the domain), a technique which takes care of both directions is used for robustness. One common method to obtain such a robust solver is using alternating line-relaxation: first along one coordinate and then along the other coordinate. Another well-known method is using line relaxation along one coordinate and coarsening only along the other coordinate.

\Cref{table:xyzebra_14_levels_e8} shows the total run-times for reducing the $L_2$ error norms by a factor of $10^8$ for $\epsilon=0.00001$ using 14 levels (about 268M finest-grid variables), where the $\kappa$-cycle is used as a stand-alone solver. The numbers in parentheses indicate the number of cycles. A comparison of several choices for the angle between the strong-diffusion direction and the $x$-coordinate is presented.
Here, the so-called alternating zebra relaxation (line Gauss Seidel in Red-Black ordering) is used first along the $x$-coordinate and then along the $y$-coordinate (only one relaxation in each direction is employed). Also, Galerkin coarsening is used for constructing the coarse-grid operator. This combination provides very good error smoothing for any angle $\phi$, and the cause for the slow convergence is inadequate coarse-grid corrections \cite{Yav98}. For every angle in the test, setting $\kappa$ to 4 yields a superior run-time.
Note that in this and in the next tests we used a machine with NVIDIA GeForce RTX 3090 GPU, which allowed us the required memory for 14 levels, whereby $\kappa > 2$ becomes advantageous.

\Cref{table:xyzebra_14_levels_cg_e8} shows run-times for the case where the $\kappa$-cycle is used as a preconditioner for Conjugate Gradients, using the same $\kappa$-cycle parameters as in \cref{table:xyzebra_14_levels_e8}.
As in the stand-alone version, for every angle in the test, setting $\kappa$ to 4 yields a superior run-time.

\Cref{table:xzebra_14_levels_1e8} shows run-times for cycles similar to those in \cref{table:xyzebra_14_levels_e8}, but with two relaxations along the x-coordinate (and no relaxation along the y-coordinate), and semi-coarsening: the grid is coarsened only along the $y$-coordinate.
As in the third solver, this combination also provides very good error smoothing for any angle.
For every angle in the test, setting $\kappa$ to 3 or 4 yields superior run-times.
Note that when semi-coarsening in 2D is used together with a $W$-cycle (as in this case), the number of operations required per cycle is no longer linear in N, but is $O(N \log N)$. This accounts for the sharp rise in run-times seen in the last row of \cref{table:xzebra_14_levels_1e8}. For other values of $\kappa$, the number of operations is linear in $N$, but the constant factor may be very large for large values of $\kappa$. This gives a relative advantage for small values of $\kappa$.

\Cref{table:xzebra_14_levels_cg_1e8} shows run-times for the case where the $\kappa$-cycle is used as a preconditioner for Conjugate Gradients, using the same $\kappa$-cycle parameters as in \cref{table:xzebra_14_levels_1e8}.
The semi-coarsening and line relaxations are expensive, especially for large values of $\kappa$, so for each of the angles, setting $\kappa$ to 2 (the $F$-cycle) yields the best time in this case.

Note that the damped Jacobi relaxations are symmetric, and the same number of relaxations is done before and after the recursive calls. Also, the full-weighting restriction operator is the adjoint of the bi-linear prolongation operator. As proved in \cite{Tatebe93}, these conditions are sufficient for the matrices of a $V$-cycle and a $W$-cycle to be symmetric and positive definite, when the start guess is zero, and therefore they are valid preconditioners for Conjugate Gradients. For $1 < \kappa< n$ the cycle is not symmetric (indeed, the $F$-cycle is described as ``not a valid preconditioner'' in \cite{Tatebe96}). Still, the MGCG algorithm converges in all the tests we checked, and, as in the example above, asymmetric $\kappa$-cycles can be better preconditioners than both $V$-cycles and $W$-cycles. 

The values of $\alpha$ and $\beta$ are system dependent. For larger ratios $\frac{\alpha}{\beta}$, the importance of smaller $\kappa$'s is more significant, because $N_{ops}$ does not change much for various $\kappa$ values, whereas the change in $N_{gpuCalls}$ may be huge, as can be inferred from \cref{table:notations}.

The code used for all experiments is publicly available at \url{https://github.com/oravnat/k_cycle} \cite{k_cycle_implementation}.

\section{Conclusions and further research} \label{sec:Conclusions} 
In this work, we have presented a new simple fixed recursive structure for multigrid algorithms, yielding a family of multigrid cycles governed by a cycle counter $\kappa$. We have derived theoretical complexity results for this algorithm, developed tools for practical prediction of run-time, and have demonstrated the new structure's utility in numerical tests, showing cases where the $\kappa$-cycle is more efficient than any one of the cycles in common use.

It would be worthwhile to explore $\kappa$-cycle performance in other problems and other computing platforms, with the aim of discovering regimes where even more significant improvement may be obtained. One such platform may be a distributed system with multiple connected computers. In such a platform the relative cost of using a large $\kappa$, and a $W$-cycle in particular, would probably be much bigger than in the GPU platform. Furthermore, distributed platforms allow solution of larger problems, and this may yield a benefit to intermediate $\kappa$'s. Also, a run-time model very similar to the one we have shown in (\ref{eq:LinearCombination}) for a GPU system may be relevant to distributed systems and other computing platforms. On the other hand, modern GPUs with reduced or partly hidden latency, should allow using larger $\kappa$ values to advantage.  

In addition, it may be interesting to test the new family of cycles on nonsymmetric problems such as the convection-diffusion equation. 
Finally, the $\kappa$-cycle should also be tested for more complex problems that are more practical for real world applications, including AMG for unstructured problems with non-constant coarsening ratio and operator-complexity.

\section{Appendix} 

We wish to estimate the number of levels, $n$, where the computation cost matches the call cost, i.e., the turning point. For a given $\kappa$,
$$
\alpha\cdot N_{gpuCalls} = \beta\cdot N_{ops} = \beta\cdot C(\kappa)\cdot(2^n-1)^2 \, .
$$
Hence,
$$
n = \log_2 \left(1+\sqrt{\left(\frac{\alpha}{\beta}\cdot \frac{N_{gpuCalls}}{C(\kappa)} \right)} \right) \, ,
$$
where, for coarsening factor 0.25, $C(\kappa) = 2\left(1-\frac{1}{3^\kappa}\right)$ depends only on $\kappa$, and $N_{gpuCalls}$ depends on $\kappa$ and $n$, but does not depend on $\alpha$ or on $\beta$. This last equation can be used to iteratively solve the problem and find the number of levels (and the number of unknowns). We have initialized $n$ to 2 and found that 20 iterations suffice for converging to $n_{tp}$.

\bibliographystyle{siam}
\bibliography{MGbib,bib_test}

\end{document}